\documentclass[12pt,a4paper]{amsart}
\usepackage{amsfonts}
\usepackage{amsthm}
\usepackage{amsmath}
\usepackage{amscd}
\usepackage{t1enc}
\usepackage{indentfirst}
\usepackage{graphicx}
\usepackage{pict2e}
\usepackage{epic}
\numberwithin{equation}{section}
\usepackage[margin=2cm]{geometry}
\usepackage{epstopdf} 
\usepackage[utf8]{inputenc} %Türkçe karakterler
\usepackage[T1]{fontenc}
\usepackage{array}
\usepackage{subcaption}

\usepackage{comment} % delate this after we are done with all the changes

\usepackage{enumitem}

\usepackage{tikz}
\usetikzlibrary{graphs}
\usetikzlibrary{graphs.standard}

\usepackage[colorlinks = true,
linkcolor = blue,
urlcolor  = blue,
citecolor = blue,
anchorcolor = blue]{hyperref}

\numberwithin{equation}{section}

\usetikzlibrary{positioning}
\tikzset{main node/.style={circle,fill=white!20,draw,minimum size=1cm,inner sep=0pt},
	}	
\usetikzlibrary {graphs.standard}	

\theoremstyle{plain}
\newtheorem{Th}{Theorem}[section]
\newtheorem{Lemma}[Th]{Lemma}
\newtheorem{Cor}[Th]{Corollary}

\theoremstyle{definition}

\newtheorem{Def}[Th]{Definition}

\newtheorem{Conj}[Th]{Conjecture}

\newtheorem{?}[Th]{Problem}
\newtheorem{Ex}[Th]{Example}
\newcommand{\rk}{\operatorname {rk}}
\newtheorem*{Nt*}{Note}

\begin{document}

\setcounter{page}{1}
\vspace{2cm}

\author{  Gökçen Dilaver$^{1*}$ , Selma Altınok$^2$ }
\title{Basis Criteria for Extending Generalized Splines}

\thanks{\noindent $^1$  Hacettepe University, Graduate School of Science and Engineering,  Beytepe, Ankara, Turkey \\
	\indent \,\,\, e-mail: gokcen.dilaver@btu.edu.tr ; ORCID: https://orcid.org/0000-0001-6055-111X. \\
		\indent  $^*$ \, Corresponding author. \\
	\indent $^2$ Hacettepe University Department of Mathematics,  Beytepe, Ankara, Turkey\\
	\indent \,\,\,  e-mail: sbhupal@hacettepe.edu.tr ; ORCID: https://orcid.org/0000-0002-0782-1587. \\
}

	\begin{abstract}
Let $R$ be a commutative ring with identity and $G$ a graph.  
Extending generalized splines  are a further extension of generalized splines  
by allowing vertex labels of $G$ to lie in varying modules rather than in a fixed ring $R$.
	Geometrically, this corresponds to the construction of equivariant cohomology 
	by Braden and MacPherson (see \cite{BM}). 
	Therefore, characterizing such splines has immediate implications in geometry, 
	particularly in the computation of equivariant cohomology.
	
	In this paper, we study extending generalized splines  as a $R$- module in which each vertex $v$ is labeled by  \( M_v = m_v R \) and each edge $e$ is labeled by \( M_e = R/r_e R \)  together with quotient $R$-module homomorphisms $M_v\to M_e $ for each vertex $v$ incident to the edge $e$, where \( R \) is a greatest common divisor domain (GCD).  
	We characterize module bases of such splines in terms of determinants so that it provides a criterion for freeness of spline modules.

	\end{abstract}

	 \maketitle
	 
	 

\section{Introduction}

Classical spline theory concerns piecewise-defined functions on polyhedral complexes that are smooth up to a certain degree. These splines play a central role in approximation theory, geometric design, and data fitting. Their algebraic properties have been deeply studied by Billera \cite{BR91, BR92}, Rose \cite{Rose95, Rose04}, and Schenck \cite{Schenck}, among others.
In \cite{BR91}, Billera and Rose described classical splines using the dual graph of a polyhedral complex, which led to the development of generalized spline theory on arbitrary graphs $G$. This framework interprets splines as vertex labelings in a ring $R$ satisfying conditions along the edges of a graph $G$ and unifies ideas from combinatorics, commutative algebra, and linear algebra.

A further development, motivated by an open problem posed by Gilbert, Polster, and Tymoczko \cite{GPT}, is the theory of \emph{extending generalized splines}. This framework, formally introduced in \cite{DA}, generalizes the classical setting by allowing each vertex to be labeled not by elements of a fixed ring, but by elements of possibly different $R$-modules. More precisely, let 
$R$ be a commutative ring with identity, and let $G=(V,E)$ a vertex labeled graph with an $R$-module $M_v$ at each vertex $v$. An \emph{extending generalized spline} is a vertex labeling $f \in \prod_{v \in V} M_v$ such that for each edge $uv \in E$, there exists an $R$-module $M_{uv}$ and $R$-module homomorphisms $\varphi_u: M_u \to M_{uv}$ and $\varphi_v: M_v \to M_{uv}$ satisfying $\varphi_u(f_u) = \varphi_v(f_v)$. The collection of vertex labelings forms an $R$-submodule:
\[
\hat{R}_G = \{ f \in \prod_{v \in V} M_v \mid \varphi_u(f_u) = \varphi_v(f_v) \text{ for all } uv \in E \}.
\]
When each $M_v = R$ and each $\varphi_u$ is the canonical quotient map to $R/\beta(uv)$ for some ideal labeling $\beta$, the definition reduces to generalized splines.

Several authors have explored bases for spline modules in various algebraic settings. Handschy, Melnick, and Reinders~\cite{HMR} studied generalized splines over $\mathbb{Z}$ on cycles and introduced \emph{flow-up classes} to form a basis for the corresponding module. In~\cite{AS2019}, Altınok and Sarıoğlan proved the existence of flow-up bases for generalized spline modules on arbitrary graphs over principal ideal domains (PIDs).

Describing of splines in terms of determinants have played a central role in analyzing basis criteria for freeness of generalized spline modules, particularly in settings involving specific graph structures and ring conditions. In~\cite{Gjoni}, Gjoni studied integer-valued generalized splines on cycle graphs $C_n$ and derived basis criteria for generalized splines using determinants associated with flow-up classes over PIDs. Using a similar approach, Mahdavi~\cite{Mahdavi} examined generalized splines over $\mathbb{Z}$ on the diamond graph $D_{3,3}$ and obtained partial results for basis conditions under certain assumptions. However, their approaches do not extend to more general settings, as flow-up bases may fail to exist when the base ring $R$ is not a PID.

Altınok and Sarıoğlan~\cite{AS2021} broadened these results to greatest common divisor (GCD) domains and subsequently provided a complete proof of the basis criterion for generalized  spline module $R_{(D_{3,3},\beta)}$ over any GCD domain. They further extended their analysis to trees, cycles, and more general diamond graphs $D_{m,n}$. Building on these developments, Calta and Rose~\cite{RC} formulated basis conditions for generalized spline modules $R_{(G,\beta)}$ on arbitrary graphs over GCD domains, under specific constraints on edge labels. Most recently, Fişekçi and Sarıoğlan~\cite{FS} resolved this open problem by establishing a general determinantal basis criterion for generalized spline modules on arbitrary graphs over GCD domains.

In our earlier work~\cite{DA}, we establish foundational results for extending generalized splines by focusing on vertices \( v \) labeled by \( M_v = m_v \mathbb{Z} \) and edges \( e \) labeled by \( M_e = \mathbb{Z}/(r_e\mathbb{Z}) \), together with quotient $\mathbb Z$-module homomorphisms \( M_v \to M_e \) for each vertex \( v \) incident to the edge \( e \).  We developed an algorithm to construct a special basis, called \emph{flow-up classes}, for extending generalized spline modules on path graphs.  This path-based algorithm was further generalized to arbitrary graphs, enabling the construction of flow-up bases in more general settings.  

In~\cite{DAreduction}, we introduced a graph reduction technique involving vertex and edge operations to explicitly construct \(\mathbb{Z}\)-module bases for extending generalized splines.  
This reduction process induces a sequence of surjective homomorphisms between the corresponding spline modules, allowing the spline space to decompose as a direct sum of certain submodules.  
The results presented here extend more generally to modules of the form \( M_v = m_v R \), where \( R \) is a PID.

When $R$ is a PID, we proved  the existence of flow-up bases of extending generalized spline modules and showed that 
the rank is exactly equal to the number of vertices, $n$ in \cite{DA}. 
However, when $R$ is not a PID, the existence of such bases is no longer guaranteed.  
This paper addresses a central question in the theory of extending generalized splines:  
\emph{Under what conditions does a given collection of splines form a basis?} We answer this question by Theorem~\ref{main1} below for a GCD $R$.

If there exists a suitable element \( \hat{Q}_G \) such that the determinant of an  \( n \)-element subset of an extending generalized spline module \( \hat{R}_G \) is an associate of \( \hat{Q}_G \) in a GCD domain \( R \),  
then the following theorem gives a basis of rank $n$ for \( \hat{R}_G \)  on an arbitrary graph $G$.

\begin{Th}{\label {main1}}
	Let  \( (G,\beta) \) be an edge-labeled graph and $R$ a GCD.  
	Suppose \( F_1, F_2, \ldots, F_n \in \hat{R}_G \) are splines.  
	If the determinant 
	\(	\lvert F_1, F_2, \ldots, F_n \rvert = u \hat{Q}_G\)
	for some unit \( u \in R \),  
	then the set \( \{F_1, F_2, \ldots, F_n\} \) forms an \( R \)-module basis of \( \hat{R}_G \).
\end{Th}

The converse of the above theorem presents significant challenges for arbitrary graphs.  
To address this, we assume the existence of a basis with certain constraints on edge and vertex labels to prove the following theorem.

\begin{Th}
	Let \( (G,\beta) \) be an edge-labeled graph over a GCD $R$ with vertex labels  
	\( m_1, m_2, \ldots, m_n \)  
	and edge labels  
	\( r_1, r_2, \ldots, r_k \),  
	where all \( m_i \) and \( r_j \) are pairwise relatively prime.  
	If  the module \( \hat{R}_G \) is free with basis  
	\( \{ F_1, F_2, \dots, F_n \} \),  
	then 
	\[
	\lvert F_1, F_2, \ldots, F_n \rvert  = u \, \hat{Q}_G
	\]
	for some unit \( u \in R \).
\end{Th}
Developing a comprehensive determinantal criterion for arbitrary graphs and rings 
remains an open and compelling direction for future research.

\begin{Conj}
	Let \( (G, \beta) \) be an edge-labeled graph over a GCD domain \( R \), 
	and let \( F_1, F_2, \ldots, F_n \in \hat{R}_{G} \) be splines.  
	If \( \{F_1, F_2, \dots, F_n\} \) forms an \( R \)-module basis of \( \hat{R}_{G} \),  
	then the determinant satisfies
	\[
	\lvert F_1, F_2, \ldots, F_n \rvert = u \hat{Q}_{G}
	\quad \text{for some unit } u \in R.
	\]
\end{Conj}

\begin{Conj}
	It would be of interest to obtain a general formula for \( \hat{Q}_{G} \).  
	In particular, before extending to arbitrary graphs, 
	one may seek explicit forms of \( \hat{Q}_{G} \) 
	for specific graph families such as trees, cycles, and diamond graphs.
\end{Conj}

%------------------------------------------------

%------------------------------------------------ 

\section{Background and Notations}

We use $P_n$ to denote a path graph with $n$ vertices, $C_n$ to denote a cycle graph with $n$ vertices and $T_n$  to denote a tree graph with $n$ vertices. Throughout the paper, we write $( \quad )$
for the greatest common divisor and $[\quad ]$ for the least common multiple.
\begin{Def}
	
	Let $R$ be a ring and $G = (V,E)$ a graph. An edge-labeling function is a map $ \beta : E \to \{\text{R-modules}\}$ which assigns an $R$-module $M_{uv}$  to each edge $uv$. We call the pair $(G, \beta)$ as an \emph{edge-labeled graph}.
	
\end{Def}

\begin{Def}

	Let $R$ be a ring and $(G,\beta)$ an edge-labeled graph. \emph{An extending generalized spline} on $(G,\beta)$ is a vertex labeling $f \in \prod_{v} M_v$ by an $R$-module $M_v$ at each vertex $v$ such that for each edge $uv$, there exists an $R$-module $M_{uv}$  together with $R$-module homomorphisms  $ \varphi_u : M_u \to M_{uv}$ and  $ \varphi_v : M_v \to M_{uv}$ satisfying the condition  $\varphi_u(f_u)=\varphi_v(f_v).$ The set of all such vertex labelings forms the module of extending generalized splines, denoted by ${\hat R}_{G}$:
	\begin{equation*}
		{\hat R}_{G}= \{f \in \prod_{v} M_v \mid \text{for each edge}\: uv, \: \varphi_u(f_u)=\varphi_v(f_v)\}.
	\end{equation*}
	This set $\hat{R}_G$ naturally carries the structure of an $R$-module.
	For elements of $\hat{R}_G $ we use either a column matrix notation with ordering from bottom to top as  	
	$$	F =
	\left( \begin{array}{c}
		f_{v_n} \\
		\vdots\\
		f_{v_1}
	\end{array} \right) \in \hat{R}_G
	$$
	\raggedright or  a vector notation as $	F= (f_{v_1},\ldots,f_{v_n}).$ 	
	
	For the sake of simplicity,  we refer to extending generalized splines as splines.

\end{Def}

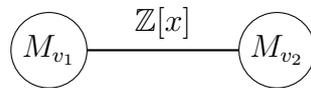
\begin{figure}[h!]
	\centering
	\begin{tikzpicture}[node distance=3cm, on grid, auto]
		\node[main node] (1) {$M_{v_1}$};
		\node[main node] (2) [right of=1] {$M_{v_2}$};
		\path[draw, thick]
		(1) edge node[above] {$\mathbb{Z}[x]$} (2);
	\end{tikzpicture}
	\caption{An edge-labeled path graph $P_2$.}
	\label{extending}
\end{figure}

\begin{Ex}
	Let $P_2$ be a path graph with modules 
	$M_{v_1} = 2\mathbb{Z}[x]$ and $M_{v_2} = 3\mathbb{Z}[x]$ 
	at the vertices $v_1$ and $v_2$, and with 
	$M_{v_1v_2} = \mathbb{Z}[x]$ on the edge, 
	as shown in Figure~\ref{extending}. 
	
	A spline on $P_2$ is a vertex labeling together with 
	identity $\mathbb{Z}[x]$-module homomorphisms:
	\[
	\begin{array}{rclcrcl}
		\varphi_{v_1} & : & 2\mathbb{Z}[x] & \longrightarrow & \mathbb{Z}[x], & \quad & 2f_{v_1} \longmapsto 2f_{v_1}, \\[4pt]
		\varphi_{v_2} & : & 3\mathbb{Z}[x] & \longrightarrow & \mathbb{Z}[x], & \quad & 3f_{v_2} \longmapsto 3f_{v_2},
	\end{array}
	\]
 such that 
	\(\varphi_{v_1}(2f_{v_1}) = \varphi_{v_2}(3f_{v_2})\).
	Then we easily observe that 
	\(
	\hat{R}_G = \langle (6,6) \rangle
	\)
	is a free $\mathbb{Z}[x]$-module.
\end{Ex}

\begin{Def}
	A spline $F = (f_{v_1}, \ldots, f_{v_n}) \in \prod_{v_i} M_{v_i}$ is called \emph{non-trivial} if $f_{v_j} \neq 0$ for at least one $j$.  
	The zero labeling $(0,\ldots,0)$ is referred to as a trivial spline.
\end{Def}

\begin{Def}
	Let  $(G,\beta)$  be an edge-labeled graph with an $R$-module $M_v$ assigned to each vertex $v$ and $\beta(uv)= M_{uv}$ to each edge $uv$. An \emph{$i$-th flow-up class} $F^{(i)}$ over $(G,\beta)$ with $1\le i \le n $ 
	is a spline for which the component $f_{v_i}^{(i)} \neq 0 $ and $f_{v_s}^{(i)} = 0 $ whenever $s < i $. The set of all $i$-th  flow-up classes is denoted by $\hat{\mathcal{F}}_i$.
\end{Def}

We include the following elementary yet useful lemma, which will play an important role in simplifying expressions involving least common multiples and greatest common divisors. As its proof is straightforward, we omit it here.

\begin{Lemma}\label{lcmgcd}
	Let $a_1, \dots, a_n, b$ be nonzero elements in a GCD domain. 
	Then the following equalities hold:
	\begin{enumerate}
		\item 
	\(	([a_1, \ldots, a_n], b)
	= 
	[(a_1, b), (a_2, b), \ldots, (a_n, b)]. \)
		
		\item 
	\(	[(a_1, \ldots, a_n), b]
	= 
	([a_1, b], [a_2, b], \ldots, [a_n, b]).
	 \)
		
		\item 
		The least common multiple of three elements satisfies
		\[
		[a_1, a_2, a_3]
		=
		\frac{a_1 a_2 a_3 \, (a_1, a_2, a_3)}
		{(a_1, a_2)\, (a_1, a_3)\, (a_2, a_3)}.
		\]
	
		\item 
		Let $\widehat{a_i}$ denote 
	\(\widehat{a_i} = a_1 a_2 \cdots a_{i-1} a_{i+1} \cdots a_n.\)
		Then the greatest common divisor of the $\widehat{a_i}$ satisfies
		\[
		(\widehat{a_1}, \widehat{a_2}, \dots, \widehat{a_n}) 
		= 
		\frac{a_1 a_2 \cdots a_n}{[a_1, a_2, \dots, a_n]}.
		\]
	\end{enumerate}

\end{Lemma}

\section{Determinant-Based Conditions on Edge-Labeled Graphs}

In this section, we consider an edge-labeled graph \((G, \beta)\) over a GCD $R$
in which each vertex \(v_i\) is assigned a label \(m_i\), corresponding to the module 
\(M_{v_i} = m_i R\) and  each edge \(e\) is assigned a label $r_e$, corrosponding to the module 
\(M_{e} = R / r_{e} R\)
for convenience.

\begin{Def}
	Let $(G,\beta)$ be an edge-labeled graph with $n$ vertices, and let $A = \{F_1, \ldots, F_n\} \subset \hat{R}_G$, where each $F_i = (f_{i1}, \ldots, f_{in})$. We represent this set in matrix form by arranging the splines as columns:
	\[
	A = \begin{bmatrix}
		f_{1n} & f_{2n} & \cdots & f_{nn} \\
		\vdots & \vdots & & \vdots \\
		f_{12} & f_{22} & \cdots & f_{n2} \\
		f_{11} & f_{21} & \cdots & f_{n1}
	\end{bmatrix}.
	\]
	We denote the determinant of this matrix by
	$
	\vert A \vert = \vert F_1, F_2, \dots, F_n \vert.$
	
\end{Def}

Throughout the remainder of this paper, we concentrate on developing determinant-based criteria for identifying bases of spline modules.  A main tool of this criteria is to provide a construction of $\hat Q_G \in R$ 
as follows:

\begin{Def}Let $(G, \beta)$ be an edge-labeled graph.  
A \emph{trail} from $v_j$ to $v_i$ is an ordered sequence
$P_{ji} = (v_j\, e_{j-1}\, v_{j-1}\, \cdots\, v_{i+1}\, e_i\, v_i)$ of vertices and edges in which no edge is repeated.  
Using the corresponding vertex and edge labels, we may write
\(P_{ji} = (m_j\, r_{j-1}\, m_{j-1}\, \cdots\, m_{i+1}\, r_i\, m_i).\)
Let $\mathcal{P}^i$ denote the set of all longest trails ending at $v_i$ (under strict inclusion).  For a trail $P_{ji}$ from $v_j$ to $v_i$, let $(P_{ji})$ denote the greatest common divisor of its edge labels,  
and let $\{(P_{ji})\}$ be the set of such gcds for all trails $P_{ji}$.
Following the notations in \cite{DA}, we define a \emph{key element} $\hat{Q}_G$ by
\[
\hat{Q}_G = \prod_{i=1}^n \mathcal{Q}^{(i)},
\]
where
\[
\mathcal{Q}^{(i)} =
\left[ m_i,\ \left\{ \left( m_j,\ \big[ \{ (P_{ji}) \} \big] \right) \;\mid\; j > i \right\},\  \{\big[ \{ (P_{si}) \} \big]  \;\mid\;  s<i \}  \right].
\]
\end{Def}
\begin{figure}[h!]
	
	\begin{center}
		\begin{tikzpicture}
			\node[main node] (1) {$m_1$};
			\node[main node] (2) [below left = 3cm of 1]  {$m_3 $};
			\node[main node] (3) [below right = 3cm of 1] {$m_2 $};
			
			\path[draw,thick]
			(1) edge node [left]{$ r_3 $} (2)
			(2) edge node [below]{$ r_2 $} (3)
			(3) edge node [right]{$ r_1 $} (1);

		\end{tikzpicture}
	\end{center}
	
	\caption{An edge-labeled cycle graph $(C_3,\beta)$} \label{c3}
\end{figure}
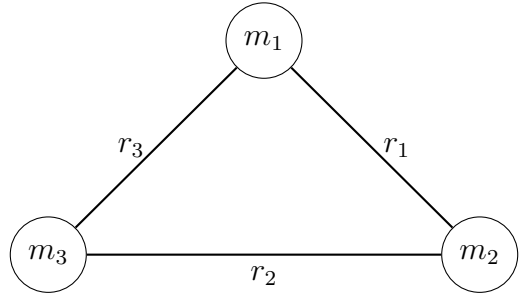

\begin{Ex}
	Let $(C_3,\beta)$ be a cycle graph with edge labels, as illustrated in Figure~\ref{c3}. By definition, we have
	\begin{align*}
		\mathcal{Q}^{1}= & [m_1,(m_2,[r_1,(r_2,r_3)]),(m_3,[r_3,(r_1,r_2)])],\\
			\mathcal{Q}^{2} = & [m_2,(m_3,[r_2,(r_1,r_3)]),[r_1,(r_2,r_3)]], \\
	\mathcal{Q}^{3} = & [m_3,r_2,r_3] 
	\end{align*}	
so that $\hat{Q}_G=\mathcal{Q}^{1}\mathcal{Q}^{2}\mathcal{Q}^{3}$.
\end{Ex}

\begin{Lemma} \label{spangraph}
	Let $(G,\beta)$ be an edge-labeled graph, and let  
	$F_1, F_2, \ldots, F_n, F \in \hat{R}_{G}$ splines with  
	\(\lvert F_1, F_2, \ldots, F_n \rvert = \hat{Q}_G\).  
	Then \( \hat{Q}_{G} F \) lies in the span of  
	\( \{F_1, F_2, \ldots, F_n\} \).
\end{Lemma}

\begin{proof}
If \( F_1, \ldots, F_n \) are linearly dependent, then  
\(\hat{Q}_G = \vert F_1,\ldots,F_n \vert =0\), and hence  
\(\hat{Q}_G F = 0\). So, it lies trivially in the span of  
\(\{F_1,\ldots,F_n\}\).
Now assume that \(F_1,\ldots,F_n\) are linearly independent and write  
\(F = (g_1,\ldots,g_n)\).  
We seek \(x_1,\ldots,x_n \in R\) such that  
\[
x_1 F_1 + \cdots + x_n F_n = \hat{Q}_G F.
\]
Let \(A\) be the matrix with columns \(F_1,\ldots,F_n\).  
Since \(\vert A \vert =\hat{Q}_G\), Cramer's rule applies to  
\[ \begin{pmatrix} f_{1n} & f_{2n} & \cdots & f_{nn} \\ \vdots & \vdots & & \vdots \\ f_{12} & f_{22} & \cdots & f_{n2} \\ f_{11} & f_{21} & \cdots & f_{n1} \end{pmatrix} \begin{pmatrix} x_1 \\ x_2 \\ \vdots \\ x_n \end{pmatrix} = \begin{pmatrix} \hat{Q}_{G} g_n \\ \vdots \\ \hat{Q}_{G} g_2 \\ \hat{Q}_{G} g_1 \end{pmatrix}. \]
 Let \(A_i\) be obtained from \(A\) by replacing its \(i\)-th column with \(\hat{Q}_{G}F\) for each $i$.  
Cramer's rule gives  
\[
x_i =  \frac{\begin{vmatrix} f_{1n} & \cdots &\hat{Q}_G g_n & \cdots & f_{nn} \\ \vdots & \vdots & & \vdots \\ f_{12} & \cdots &\hat{Q}_G g_2 & \cdots & f_{n2} \\ f_{11} & \cdots & \hat{Q}_G g_1 & \cdots & f_{n1} \end{vmatrix}}{\hat{Q}_G} = \begin{vmatrix} f_{1n} & \cdots & g_n & \cdots & f_{nn} \\ \vdots & \vdots & & \vdots \\ f_{12} & \cdots & g_2 & \cdots & f_{n2} \\ f_{11} & \cdots &  g_1 & \cdots & f_{n1} \end{vmatrix}.
\]
so that \(x_i\in R\).
Thus \( \hat{Q}_G F \) lies in the span of \(\{F_1,\ldots,F_n\}\).

\end{proof}

\begin{Lemma}\label{combination}
	
	Let $(G,\beta)$ be an edge-labeled  graph, and let $\{F_1,F_2, \dots, F_n\}$ form a basis for $\hat{R}_G$. If $H_1, H_2, \dots, H_n$ are  linear combinations of $F_1, F_2, \dots, F_n$, then $\vert F_1, F_2, \dots,F_n \vert $ divides $ \vert H_1, H_2,\dots, H_n \vert $.
	
\end{Lemma}

\begin{proof}
	See Lemma 5.1.4 in \cite{Gjoni}. 
\end{proof}

\begin{Cor}\label{twobasis}
	Let $(G,\beta)$ be an edge-labeled graph. If $\{F_1,F_2, \dots,F_n\}$ is a basis for $\hat{R}_G$ and $\{H_1,H_2,\dots,H_n\}$ is another basis  for $\hat{R}_G$, then $\vert F_1, F_2, \dots,F_n \vert = u \vert H_1, H_2, \dots,H_n \vert $ where $u\in R$ is a unit.	
\end{Cor}

Now we prove one of the main theorems over a GCD $R$ as follows.

\begin{Th}\label{GCD basis}
	Let  \( (G,\beta) \) be an edge-labeled graph and \(\{F_1, F_2, \ldots, F_n \}\subset \hat{R}_G \).  
	If the determinant \( \lvert F_1, F_2, \ldots, F_n \rvert = u \hat{Q}_G \) for some unit \( u \in R \),  
	then the set \( \{F_1, F_2, \ldots, F_n\} \) forms an \( R \)-module basis of \( \hat{R}_G \).
\end{Th}

\begin{proof}
	Suppose that 
	\( \lvert F_1, F_2, \ldots, F_n \rvert = u \hat{Q}_G \) for some unit \( u \in R \).
	Since \(\hat{Q}_{G} \neq 0\), the splines \(F_1,F_2,\dots,F_n\) are linearly independent. 	
	Let \(F \in \hat{R}_{G}\) be an arbitrary spline. 
	We wish to show that \(F\) can be expressed as a linear combination of 
	\(\{F_1,F_2,\dots,F_n\}\).
	By Lemma~\ref{spangraph}, we know that \(\hat{Q}_{G}F\) lies in the span of \(\{F_1,F_2,\dots,F_n\}\). 
	Therefore, there exist \(x_1,x_2,\dots,x_n \in R\) such that 
	\[
	\hat{Q}_{G}F = x_1 F_1 + x_2 F_2+\cdots + x_n F_n.
	\]
To determine the coefficients $x_i$, we write
\begin{align*}
	u x_1 \hat{Q}_{G} = & x_1 \vert F_1, F_2,\dots,F_n \vert = \vert x_1F_1, F_2,\dots,F_n \vert = \vert (x_1 F_1+ x_2 F_2+ \cdots + x_n F_n), F_2,\dots, F_n \vert \\
	= & \vert \hat{Q}_{G}F_1, F_2,\dots, F_n \vert = \hat{Q}_{G} \vert F_1, F_2, \dots, F_n \vert .
\end{align*}
so that $ x_1 = c_1\vert F_1, F_2, \dots, F_n \vert $ for some unit $c_1 \in R$.
Similarly, we can obtain:  
\begin{equation*}
	x_2 = c_2 \hat{Q}_{G}, \dots, x_n= c_n \hat{Q}_{G}
\end{equation*}
for some units $c_2,\dots, c_n \in R$. Substituting back, we find:
\begin{align*}
	\hat{Q}_{G}F= & x_1 F_1 + x_2 F_2 + \cdots + x_n F_n  \\
	= &(c_1 \hat{Q}_{G})F_1 + (c_2 \hat{Q}_{G}) F_2 + \cdots + (c_n \hat{Q}_{G})F_n  = \hat{Q}_{G} ( c_1 F_1 +  c_2 F_2 + \cdots +c_n F_n).
\end{align*}
	Hence, $F$ lies in the span of  $\{F_1,F_2, \dots, F_n\}$, and so $\{F_1,F_2, \dots, F_n\}$ forms a basis.

\end{proof}

If $R$ is a PID, we proved in \cite{DA} that there exists a flow-up basis ${\{F_1,\ldots ,F_n\}}$ for an extending generalized spline module $\hat{R}_G$ such that	\( \lvert F_1, F_2, \ldots, F_n \rvert = u \hat{Q}_G \) for some unit $u$ and its rank equals the number of vertices $n$.  
When $R$ is a GCD domain which is not a PID, there exists a free spline module that does not have a flow-up basis. The following example illustrates this.

\begin{figure}[h]
	\begin{center}
		\begin{tikzpicture}
			\node[main node] (1) {$m_3$};
			\node[main node] (2) [right = 2cm of 1]  {$m_2$};
			\node[main node] (3) [above = 2cm of 1]  {$m_1$};
			\node[main node] (4) [left = 2cm of 1]  {$m_4$};
			
			\path[draw,thick]
			(1) edge node [above]{$ r_2 $} (2)
			(1) edge node [right]{$ r_1 $} (3)
			(1) edge node [above]{$  r_3 $} (4);
			
		\end{tikzpicture}
	\end{center}
	
	\caption{Tree graph $T_4$} \label{counterex}
	
\end{figure}
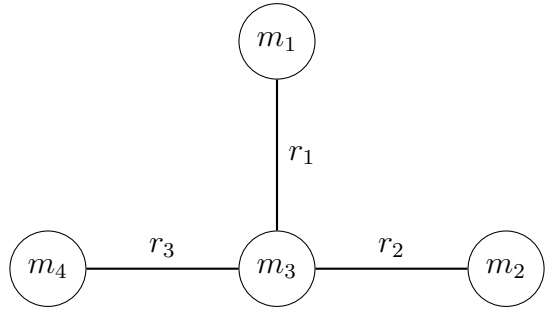

\begin{Ex}
	We consider the edge-labeled tree graph $(T_4, \beta)$ with $\mathbb{Z}[x,y]$-modules 
	$M_{v_i} = m_i \mathbb{Z}[x,y]$ at each vertex $v_i$, where
	\[
	m_1 = x, \quad m_2 = y^2, \quad m_3 = x + y, \quad m_4 = xy,
	\]
	and 
	$\beta(e) = \mathbb{Z}[x,y] / r_{e} \mathbb{Z}[x,y]$ at each edge $e$ where
	\[
	r_1 = x^2 + y, \quad r_2 = x^2, \quad r_3 = y,
	\]
	as in Figure~\ref{counterex}. First of all, we want to show that the longest-path method in \cite {DA} in order to determine the smallest nonzero entry $f_{v_2}^{(2)}$ of $F^{(2)}$ does not work here when $R$ is not PID. Secondly we want to show that there is no flow-up basis even though the spline module is free.		
		
Now if we apply the longest-path method, we obtain  \( f_{v_2}^{(2)} = y^2 \).  By using spline conditions, it can easily be proved that
 there is no $f_{v_3}^{(2)}$ satisfying
	\[
	f_{v_3}^{(2)} - y^2 \in x^2 \mathbb{Z}[x,y], \qquad 
	f_{v_3}^{(2)} \in (x+y)\mathbb{Z}[x,y], \qquad 
	f_{v_3}^{(2)} \in (x^2+y)\mathbb{Z}[x,y]
	\]
	since $(x+y)\mathbb{Z}[x,y] \cap (x^2+y)\mathbb{Z}[x,y] = (x+y)(x^2+y)\mathbb{Z}[x,y]$.
Hence there is no flow-up class $F^{(2)}$ with \( f_{v_2}^{(2)} = y^2 \).
There exists a set of flow-up classes which is not a basis for \(\hat{R}_{G}\):
\[
A =
\left\{
\begin{pmatrix}
	0\\ 0\\ 0\\ x^{3}+xy
\end{pmatrix},
\begin{pmatrix}
	0\\ 0\\ x^{2}y^{2}\\ 0
\end{pmatrix},
\begin{pmatrix}
	0\\ (x+y)(x^{2}+y)x^{2}y\\ 0\\ 0
\end{pmatrix},
\begin{pmatrix}
	xy\\ 0\\ 0\\ 0
\end{pmatrix}
\right\}.
\]
  Indeed, the spline 
\(F = (x^{2}y^{2}+x^{2}y,\; -y^{3},\; x^{2}y-y^{3},\; 0)\)
cannot be written as a \(\mathbb{Z}[x,y]\)-linear combination of the elements of \(A\).  
Comparing the second entries in
\(F=a_{1}A_{1}+\cdots+a_{4}A_{4}\) it gives  
\(a_{2}x^{2}y^{2}=-y^{3}\), hence \(a_{2}=-y/x^2\), which is not in \(\mathbb{Z}[x,y]\).  
Thus \(F\notin \mathrm{span}_{\mathbb{Z}[x,y]}(A)\).
In fact, there is no flow-up basis at all. 
To see this, we assume that
\[
B=\{F^{(1)}, F^{(2)}, F^{(3)}, F^{(4)}\}
\]
is a flow-up basis and derive a contradiction.
Let
\[
F_1 = (x^{3}+xy,\, 0,\, 0,\, 0)
\qquad\text{and}\qquad
F_2 = (x^{2}y^{2}+x^{2}y,\; -y^{3},\; x^{2}y-y^{3},\; 0)
\]
be splines.  
Since \(B\) is a basis, both \(F_1\) and \(F_2\) must be expressible as 
\(\mathbb{Z}[x,y]\)-linear combinations of the elements of \(B\).  
Comparing the first entries in these expansions imposes conditions on the 
coefficients which cannot be simultaneously satisfied.  
Indeed, the first entries give
\[
x^{3}+xy = b_1\,f^{(1)}_{v_1}
\qquad\text{and}\qquad
x^{2}y^{2}+x^{2}y = b_2\,f^{(1)}_{v_1}
\]
for some \(b_1,b_2 \in \mathbb{Z}[x,y]\).
Thus \(f^{(1)}_{v_1}\) must divide both polynomials.
However, in \(\mathbb{Z}[x,y]\), the polynomials  
\(x^{3}+xy\) and \(x^{2}y^{2}+x^{2}y\) are not associates, and therefore they 
cannot have a common divisor other than a unit.  
Since \(f^{(1)}_{v_1}\) is a non-unit by construction, this is impossible.  
Thus, no such coefficients \(b_1,b_2\) exist. Therefore this is a contradiction by assuming the existence of flow-up basis.

	Finally, by using the matrix $M_G$ defined in \cite{DA} and computing in \texttt{Macaulay2}, 
	we obtain a module basis
	\[
	B =
	\left\{
	\begin{pmatrix}
		0 \\[2pt] 0 \\[2pt] 0 \\[2pt] x^3 + xy
	\end{pmatrix},
	\begin{pmatrix}
		0 \\[2pt] -xy^2 - y^3 \\[2pt] -xy^2 - y^3 \\[2pt] x^2y^2 - xy^2
	\end{pmatrix},
	\begin{pmatrix}
	0\\ x^{2}y+xy^{2}\\ xy^{2}\\ x^{2}y+xy^{2}
	\end{pmatrix},
	\begin{pmatrix}
		xy \\[2pt] 0 \\[2pt] 0 \\[2pt] 0
	\end{pmatrix}
	\right\}.
	\]
	By definition of $\hat Q_{T_4}$ we obtain $\hat Q_{T_4} = x^4y^4(x+y)(x^2+y)$ which is
	 equal to $-\lvert B \rvert$.

\end{Ex}

The converse of Theorem~\ref{GCD basis} is not easy to establish.  
Thus, instead of treating the general case, we focus on a special situation in which 
the vertex labels \(m_1, \ldots, m_n\) and the edge labels \(r_1, \ldots, r_k\) 
are pairwise coprime.

\begin{Lemma}\label{divides}	Let $(G,\beta)$ be an edge-labeled  graph and $m_1,\ldots,m_n,r_1,\ldots,r_k$ pairwise coprime. Then
		\[
	\hat{Q}_G = \left(\prod_i^n m_i\right)\left(\prod_i^k r_i\right)=[m_1, \dots, m_n, r_1, \dots, r_k].
	\]
	
	\end{Lemma}
\begin{proof}
	By definition, we have
	\(\hat{Q}_G = \prod_{i=1}^n \mathcal{Q}^{(i)},\)
	where
	\[
	\mathcal{Q}^{(i)} =
	\left[
	m_i,\ 
	\left\{
	\big( m_j,\ [ \{ P_{ji} \} ] \big) \;\middle|\; j > i
	\right\},
	\ \left\{
	[ \{ P_{si} \} ] \;\middle|\; s < i
	\right\}
	\right].
	\]
	Since all \( m_i \) and \( r_j \) are pairwise relatively prime, it follows that
	\[
	\hat{Q}_G = \left( \prod_{i=1}^{n} m_i \right)
	\left( \prod_{i=1}^{k} r_i \right)
	\] which is $[m_1, \dots, m_n, r_1, \dots, r_k]$.
\end{proof}

\begin{Th}\label{arbitrary-divides}
	Let $(G,\beta)$ be an edge-labeled  graph and $m_1,\ldots,m_n,r_1,\ldots,r_k$ pairwise coprime
	Then $\hat{Q}_G$ divides the determinant of the elements $F_1,\ldots ,F_n$ of $\hat{R}_G$. 
\end{Th}

\begin{proof}
	Let $F_i = (f_{i1}, f_{i2}, \dots, f_{in}) \in \hat{R}_G$ be splines, and we consider the determinant of the following matrix M
	\[
	\vert M \vert=
	\begin{vmatrix}
		f_{1n} & f_{2n} & \cdots & f_{nn} \\
		\vdots & \vdots & \ddots & \vdots \\
		f_{12} & f_{22} & \cdots & f_{n2} \\
		f_{11} & f_{21} & \cdots & f_{n1}
	\end{vmatrix}
	=
	\begin{vmatrix}
		f_{1n} & f_{2n} & \cdots & f_{nn} \\
		\vdots & \vdots & \ddots & \vdots \\
		f_{12} & f_{22} & \cdots & f_{n2} \\
		f_{11} - f_{12} & f_{21} - f_{22} & \cdots & f_{n1} - f_{n2}
	\end{vmatrix}.
	\]
	By definition, each component $f_{ij} \in m_j R$ for all $i,j = 1,2,\dots,n$. It follows that $m_j$ divides $\vert M \vert$. 
	Without loss of generality, let $r_1$ be the edge label between $ v_1$ and $v_2$. 
	The spline condition implies that $f_{11} - f_{12} \in r_1 R$ 
	so that $r_1 \mid \vert M\vert$. A similar argument shows that for any $\ell$, $r_\ell \mid \vert M\vert$. 
	Hence, by Lemma ~\ref{divides}, we have
	\[
	\hat{Q}_G = [m_1, \dots, m_n, r_1, \dots, r_k] \mid \vert M\vert.
	\]
\end{proof}

\begin{Th}
	Let \( (G,\beta) \) be an edge-labeled graph with vertex labels  
	\( m_1, m_2, \ldots, m_n \)  
	and edge labels  
	\( r_1, r_2, \ldots, r_k \) pairwise relatively prime.  
	If  the module \( \hat{R}_G \) is free with basis  
	\( \{ F_1, F_2, \dots, F_n \} \),  
	then 
	\[
	\lvert F_1, F_2, \ldots, F_n \rvert  = u \, \hat{Q}_G
	\]
	for some unit \( u \in R \).
\end{Th}

\begin{proof}
Assume that \( \{F_1, F_2, \ldots, F_n\} \) forms an \( R \)-module basis.  
Then, by Theorem~\ref{arbitrary-divides} we have
$$\lvert F_1, F_2, \ldots, F_n \rvert = u \hat{Q}_{G}$$
for some \( u \in R \).  
It therefore suffices to show that \( u \) is a unit.
We collectively denote \( m_i \) and \( r_j \) by \( l_i \) and \( l_{n+j} \), respectively.
For each \( i \), set
\[
\hat{l}_i = l_1 l_2 \cdots l_{i-1} \, l_{i+1} \cdots l_{n+k}.
\]
By Lemma~\ref{lcmgcd}, we have
\[
(\hat{l}_1, \hat{l}_2, \dots, \hat{l}_{n+k})
= \frac{l_1 l_2 \cdots l_{n+k}}{[l_1, l_2, \dots, l_{n+k}]}.
\]
Since all \( l_i \) and \( l_j \) are pairwise relatively prime, it follows that
\([l_1, l_2, \dots, l_{n+k}] = l_1 l_2 \cdots l_{n+k}.\)
Hence,
\((\hat{l}_1, \hat{l}_2, \dots, \hat{l}_{n+k}) = 1.\)
We now construct matrices
\[
A^{(i)} =
\begin{bmatrix}
	A^{(i)}_1 & A^{(i)}_2 & \cdots & A^{(i)}_n 
\end{bmatrix}.
\]
In other words,
\[
A^{(i)} =
\begin{bmatrix}
	a_{1n}^{(i)} & a_{2n}^{(i)} & \cdots & a_{nn}^{(i)} \\[4pt]
	\vdots & \vdots & & \vdots \\[4pt]
	a_{12}^{(i)} & a_{22}^{(i)} & \cdots & a_{n2}^{(i)} \\[4pt]
	a_{11}^{(i)} & a_{21}^{(i)} & \cdots & a_{n1}^{(i)}
\end{bmatrix},
\]
for \( i = 1, \dots, n+k \).  
For a fixed \( i \in \{1, \dots, n\} \), we define the entries of the columns \( A^{(i)}_j \) for \( j = 1, \dots, n \) by
\[
a_{j_1 j_2}^{(i)} =
\begin{cases}
	\hat Q_G, & \text{if } i = j_1 = j_2, \\[4pt]
	\hat{l}_i, & \text{if } i \neq j_1 = j_2, \\[4pt]
	0, & \text{if } i \neq j_1 \neq j_2.
\end{cases}
\]
It is easy to see that each column \( A^{(i)}_j \) for \( i = 1, \dots, n \) defines a spline.  

Next, we define the matrices \( A^{(i)} \) for \( i = n+1, \dots, n+k \).  
Without loss of generality, let \( r_1 \) be the edge connecting \( v_1 \) and \( v_2 \).  
Then we can write
\[
A^{(n+1)} =
\begin{bmatrix}
	0 & 0 & 0 & \cdots & \hat{l}_{n+1} \\[4pt]
	\vdots & \vdots & & \vdots \\[4pt]
	0 & 0 & \hat{l}_{n+1} & \cdots & 0 \\[4pt]
	\hat{l}_{n+1} & \hat{Q}_G & 0 & \cdots & 0 \\[4pt]
	\hat{l}_{n+1} & 0 & 0 & \cdots & 0
\end{bmatrix}.
\]
Again, each column \( A^{(n+1)}_j \) for \( j = 1, \dots, n \) is a spline, and a similar construction applies to any other edge \( r_i \).  
Thus, we can construct all matrices \( A^{(i)} \) for \( i = 1, \dots, n+k \)
whose each column \( A^{(i)}_j \) for \( i = 1, \dots, n+k \) is a spline. It follows from Lemma~\ref{combination} that it lies in the span of the \( F_i \)'s.
Because
\[
\vert F_1, F_2, \ldots, F_n \vert = u \hat{Q}_G
\quad \text{for some } u \in R,
\]
and $\vert A^{(i)} \vert$ divides $\vert F_1, F_2, \ldots, F_n \vert$ we obtain
\[
 \vert A^{(i)} \vert = t_i u \hat{Q}_G
\quad \text{for some } t_i \in R.
\]
On the other hand, by the determinant of $A^{(i)}$, we have
\[
\vert A^{(i)} \vert  = (\hat{l}_{i})^{n-1} \hat{Q}_G.
\]
Hence \( u \) divides \( \hat{l}_{i} \) for all \( i \).  
Since \( (\hat{l}_1, \hat{l}_2, \dots, \hat{l}_{n+k}) = 1 \), it follows that \( u \) is a unit.

\end{proof}

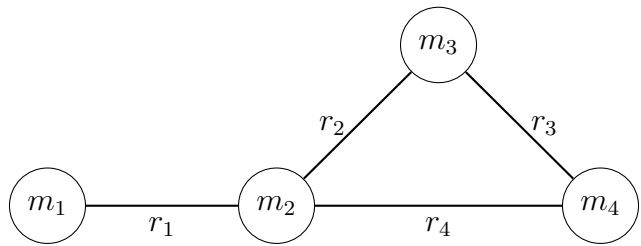
\begin{figure}[h]
	
	\begin{center}
		\begin{tikzpicture}
			\node[main node] (1) {$m_3$};
			\node[main node] (2) [below left = 2cm of 1]  {$m_2$};
			\node[main node] (3) [below right = 2cm of 1] {$m_4$};
				\node[main node] (4) [left = 2cm of 2] {$m_1$};
			
			\path[draw,thick]
			(1) edge node [left]{$r_2 $} (2)
			(2) edge node [below]{$r_4$} (3)
			(2) edge node [below]{$r_1$} (4)
			(3) edge node [right]{$r_3$} (1);

		\end{tikzpicture}
	\end{center}
	
	\caption{An edge-labeled graph $(G,\beta)$} \label{graphex}
\end{figure}

\begin{Ex}
Consider the edge-labeled graph \( (G, \beta) \) shown in Figure~\ref{graphex},  
with vertex labels \( m_i \) and edge labels \( r_j \),  
where all \( m_i \) and \( r_j \) are pairwise relatively prime.
Let \( \{ F_1, F_2, F_3, F_4 \} \) be a basis for \( \hat{R}_G \).  
By Theorem~\ref{arbitrary-divides}, we have
\(\lvert F_1, F_2, F_3, F_4 \rvert = u \hat{Q}_{G}\)
for some \( u \in R \), where
\[
\hat{Q}_{G} = \prod_{i=1}^{4} (m_i \cdot r_i).
\]
We denote the vertex and edge labels \( m_i \) and \( r_j \) by \( l_i \) and \( l_{n+j} \), respectively.  
For each \( i \), define
\[
\hat{l}_i = l_1 l_2 \cdots l_{i-1} \, l_{i+1} \cdots l_{8}.
\]
Since all \( l_i \) and \( l_j \) are pairwise relatively prime, by Lemma~\ref{lcmgcd} we have
\[
(\hat{l}_1, \hat{l}_2, \dots, \hat{l}_{8}) = 1.
\]
Now consider the following matrices:
\[
A^{(1)} =
\begin{bmatrix}
	0 & 0 & 0 & \hat{l}_{1} \\[4pt]
	0 & 0 & \hat{l}_{1} & 0 \\[4pt]
	0 & \hat{l}_{1} & 0 & 0 \\[4pt]
	\hat{Q}_G & 0 & 0 & 0
\end{bmatrix},
\quad
A^{(2)} =
\begin{bmatrix}
	0 & 0 & 0 & \hat{l}_{2} \\[4pt]
	0 & 0 & \hat{l}_{2} & 0 \\[4pt]
	0 & \hat{Q}_G & 0 & 0 \\[4pt]
	\hat{l}_{2} & 0 & 0 & 0
\end{bmatrix},
\]
\[
A^{(3)} =
\begin{bmatrix}
	0 & 0 & 0 & \hat{l}_{3} \\[4pt]
	0 & 0 & \hat{Q}_G& 0 \\[4pt]
	0 & \hat{l}_3  & 0 & 0 \\[4pt]
	\hat{l}_{3} & 0 & 0 & 0
\end{bmatrix},
\quad 
A^{(4)} =
\begin{bmatrix}
	0 & 0 & 0 & \hat{Q}_G \\[4pt]
	0 & 0 & \hat{l}_{4} & 0 \\[4pt]
	0 & \hat{l}_{4} & 0 & 0 \\[4pt]
	\hat{l}_{4} & 0 & 0 & 0
\end{bmatrix},
\]
\[A^{(5)} =
\begin{bmatrix}
	0 & 0 & 0 & \hat{l}_{5} \\[4pt]
	0 & 0 & \hat{l}_{5} & 0 \\[4pt]
	\hat{l}_{5} & \hat{Q}_G & 0 & 0 \\[4pt]
	\hat{l}_{5} & 0 & 0 & 0
\end{bmatrix},
\quad
A^{(6)} =
\begin{bmatrix}
	0 & 0 & 0 & \hat{l}_{6} \\[4pt]
	\hat{l}_{6} & 0 &0 & 0 \\[4pt]
	\hat{l}_{6} & \hat{Q}_G  & 0 & 0 \\[4pt]
	0 & 0 & \hat{l}_6 & 0
\end{bmatrix},\]
\[A^{(7)} =
\begin{bmatrix}
	\hat{l}_{7} & 0 & 0 & 0 \\[4pt]
	\hat{l}_{7} & 0 & 0 & \hat{Q}_G \\[4pt]
	0 & 0 & \hat{l}_{7}& 0 \\[4pt]
	0 & \hat{l}_{7} & 0 & 0
\end{bmatrix},
\quad
A^{(8)} =
\begin{bmatrix}
		\hat{l}_{8} & 0 & 0 &\hat{Q}_G \\[4pt]
	0 & 0 & \hat{l}_8& 0 \\[4pt]
		\hat{l}_{8} & 0  & 0 & 0 \\[4pt]
 0& \hat{l}_8 & 0 & 0
\end{bmatrix}.\]
Each column of $A^{(i)}$ is a spline for \( \hat{R}_G \). Thus, by Lemma~\ref{combination},
\( \vert F_1, F_2, F_3,F_4 \vert  = u \hat{Q}_G\) divides 
\(\vert A^{(i)} \vert  = (\hat{l}_{i})^{3} \hat{Q}_G.\)
Hence \( u \) divides \( \hat{l}_{i} \) for all \( i \).  
Since \( (\hat{l}_1, \hat{l}_2, \dots, \hat{l}_{8}) = 1 \), it follows that \( u \) is a unit.

\end{Ex}

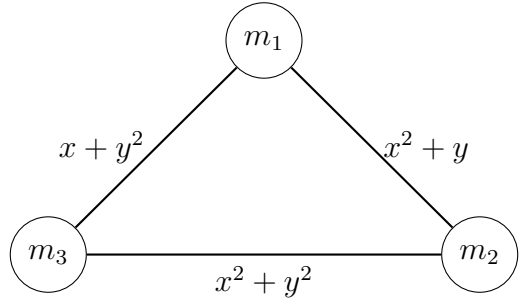
\begin{figure}[h]
	
	\begin{center}
		\begin{tikzpicture}
			\node[main node] (1) {$m_1$};
			\node[main node] (2) [below left = 3cm of 1]  {$m_3$};
			\node[main node] (3) [below right = 3cm of 1] {$m_2$};
			
			\path[draw,thick]
			(1) edge node [left]{$x+y^2 $} (2)
			(2) edge node [below]{$x^2+y^2$} (3)
			(3) edge node [right]{$x^2+y$} (1);

		\end{tikzpicture}
	\end{center}
	
	\caption{An edge-labeled cycle graph $(C_3,\beta)$} \label{cycleex} 
\end{figure}

\begin{Ex}
	Consider the edge-labeled cycle graph $(C_3, \beta)$ with $R$-modules 
	$M_{v_i} = m_i \mathbb{Q}[x,y]$ at each vertex $v_i$, where
	\[
	m_1 = x, \quad m_2 = y, \quad m_3 = x + y
	\]
	as in Figure \ref{cycleex}.	By employing the matrix $M_G$ introduced in \cite{DA} and performing the necessary computations in \texttt{Macaulay2}, 
	we obtain the following module basis:
	\[
	B =
	\left\{
	\begin{pmatrix}
		x^2y+xy^2 \\[2pt]
		x^2y+xy^2 \\[2pt]
		x^2y+xy^2
	\end{pmatrix},
	\begin{pmatrix}
		x^3+x^2+xy^3-2xy^2+xy+y^4-y^3 \\[2pt]
		x^3y-x^2y^2+2xy^3-3xy^2+xy-y^3-y^2 \\[2pt]
		x^3+x^2y+2xy^3-4xy^2+2xy
	\end{pmatrix},
	\begin{pmatrix}
		x^4-x^3+3xy^2-y^4+2y^3\\[2pt]
		-x^2y+4xy^2+y^3 \\[2pt]
		x^4-2x^3-x^2y^2+4xy^2-2xy 	
	\end{pmatrix}
	\right\}.
	\]
	It follows that
	\[
\lvert B \rvert = 2xy(x+y)(x+y^{2})(x^{2}+y)(x^{2}+y^{2}),
	\]
	which coincides with $2	\hat Q_{C_3}$.
	
\end{Ex}

\section{Determinantal Criteria Over PID}

If \(R\) is a PID, we proved in \cite{DA} that flow-up bases always exist and the rank
equals the number of vertices \(n\).  
In contrast, when \(R\) is not a PID, flow-up bases may not need to exist even though when the
spline module is free.  

In this section, we focus on the PID case and describe the precise relationship
between generalized splines and extending generalized splines.  
We show that, when \(R\) is a PID, the existence of a minimal flow-up basis yields
a determinantal criterion that holds for arbitrary graphs.  

\begin{Def}
	Let $F^{(i)}=(0,\dots,0,f_{v_i}^{(i)}, \dots, f_{v_n}^{(i)}) \in \hat{\mathcal{F}_i}$   be a flow-up class with $f_{v_i}^{(i)} \neq 0$ for $i=1,2,\ldots,n$ and $f_{v_s}^{(i)} =0$ for all $s <i$. We define the leading term of $F^{(i)}, \mathrm{LT}(F^{(i)})=f_{v_i}^{(i)}$, the first nonzero entry of $F^{(i)}$. 
	The set of all leading terms of splines in $\hat{\mathcal{F}_i}$ together with a trivial spline, $<\mathrm{LT}(\hat{\mathcal{F}_i})>,$ forms an ideal of $R$.
\end{Def}

\begin{Def}
	A flow-up class $F^{(i)} $ is called a minimal element of $ \hat{\mathcal{F}_i}$ if $\mathrm{LT}(F^{(i)})$ is a generator of the ideal  $\mathrm{LT}(\hat{\mathcal{F}_i})$.
\end{Def}

\begin{Def}
	\emph{A minimum generating set} for a $\mathbb{Z}$-module $\hat{R}_G$ is a spanning set of splines with the smallest possible number of elements. The size of a minimum generating set is called  \emph{rank} and is denoted by $\rk \hat{R}_G$.
\end{Def}

\begin{Th}[The Chinese Remainder Theorem] \label{CRT}
	Let $R$ be a PID and $x, a_1,\ldots, a_n, b_1, \ldots, b_n \in R$. Then the system
	\begin{align*}
		x & \equiv a_1 \mod b_1 \\
		x & \equiv a_2 \mod b_2 \\
		& \quad \quad \vdots\\
		x & \equiv a_n \mod b_n
	\end{align*}
	has a solution if and only if $a_i \equiv a_j \mod (b_i, b_j )$ for all $i, j \in \{1,\dots,n\}$ with $i \neq j$.
	The solution is unique modulo $[b_1,\dots,b_n]$.

\end{Th}

\begin{Th}
	Let  $(G,\beta)$ be an edge-labeled graph  and  $F^{(i)} = (0,\dots,0,f_{v_i}^{(i)},\dots,f_{v_n}^{(i)})$  a flow-up class where $f_{v_i}^{(i)} \neq 0$ for $i>1$ and $f_{v_s}^{(i)} = 0$ for all $s<i$.  
	Consider a trail $P_{ji}$ from $v_j$ to $v_i$. Let $(P_{ji})$  be the greatest common divisor of the edge labels of $P_{ji}$ and $\{(P_{ji})\}$ denote the set of greatest common divisors of the edge labels of all trails $P_{ji}$.  
	Then $f_{v_i}^{(i)}$ is divisible by
	\[
	\left[ m_i,\ \left\{ \left( m_j,\ \big[ \{ (P_{ji}) \} \big] \right) \;\vert\; j > i \right\},\  \{\big[ \{ (P_{si}) \} \big]  \;\mid\;  s<i \}  \right].
	\]
	
\end{Th}

\begin{proof}
	Assume that $f_{v_s}^{(i)} = 0$ for all $s<i$ and $f_{v_i}^{(i)}\ne 0$.  
	For each an index $j >i$ and $s<i$ by using spline conditions
	we can obtain the following system of congruences:
	\[
	\begin{aligned}
		f_{v_i}^{(i)} &\equiv f_{v_j}^{(i)} 
		&&\mod\big[\{(P_{ji}) \}\big] \\[6pt]
		f_{v_i}^{(i)} &\equiv 0 
		&&\mod\big[\{(P_{si} )\}\big] \\[6pt]
		f_{v_j}^{(i)} &\equiv 0 
		&&\mod m_j \\[6pt]
		f_{v_i}^{(i)} &\equiv 0 
		&&\mod m_i.
	\end{aligned}
	\]
	By Theorem \ref{CRT} (CRT),  the system has a solution for $f_{v_{i}}^{(i)}$  if and only if \[f_{v_{i}}^{(i)} \equiv 0 \mod \left[ m_i,\ \left\{ \left( m_j,\ \big[ \{ (P_{ji}) \} \big] \right) \;\mid\; j >i \right\},\  \big[ \{ (P_{si}) \}_{s<i} \big] \right].\]  
	This means that  $f_{v_i}^{(i)}$ is divisible by
	\[
	\left[ m_i,\ \left\{ \left( m_j,\ \big[ \{ (P_{ji}) \} \big] \right) \;\mid\; j >i \right\},\  \{ (P_{si}) \}_{s<i}  \right].
	\]
\end{proof}

\begin{Cor}\label{det}
Let $(G,\beta)$ be an edge-labeled graph. 
For each $i=1,\dots,n$, there exists a flow-up class
\(F^{(i)} = (0,\dots,0,\,f^{(i)}_{v_i},\,\dots, f^{(i)}_{v_n})\)
where
\[
f^{(i)}_{v_i}
=
\left[
\,m_i,\ 
\big\{ (m_j,\,[\{(P_{ji})\}]) \;\vert\; j>i \big\},\
\big\{ [\{(P_{si})\}] \;\vert\; s<i \big\}
\right].
\]
These flow-up classes form a basis satisfying
\[
\bigl|\, F^{(1)},\, F^{(2)},\,\dots,\,F^{(n)} \,\bigr| 
= \pm \,\widehat{Q}_{G}.
\]

\end{Cor}

\begin{Cor}\label{Q-divides}
	Let  $(G,\beta)$ be an edge-labeled graph.  
	Then $\hat{Q}_G$ is divisible by 
	$\big[ \{ (P_{si}) \}_{s<i} \big]   $.
\end{Cor}
\begin{Nt*}
	The expression $\big[ \{ (P_{si}) \}_{s<i} \big]$ corresponds to the least common multiple of the greatest common divisors of the zero trails 
	defined in~\cite{AS2021}.  
	In the construction of $\hat{Q}_G$, each vertex module is taken to be 
	$M_v = m_v R$, and each edge $uv$ is labeled with the module 
	$M_{uv} = R / r_{uv} R$, equipped with the natural 
	quotient homomorphism $\varphi_u : M_u \to M_{uv}$.  
	
	Now, consider the specialization where all vertex labels are set to be $m_i = 1$
	so that $M_v = R$ for all vertices $v$.  
	Under this setting, each map $\varphi_u$ reduces to the canonical quotient map
	\[
	\varphi_u : R \longrightarrow R / \beta(uv),
	\]
	where $\beta$ denotes the edge-labeling.  
	With this specialization, our extending generalized spline framework 
	coincides exactly with the classical generalized spline setting studied 
	in~\cite{AS2021}.  
	Consequently, it follows that
	$
	Q_G = \hat{Q}_G.$
\end{Nt*}

\begin{Lemma}
	Let $(G,\beta)$ be an edge-labeled graph and $Q_G$ defined as in 
	Altınok and Sarıoğlan~\cite{AS2021}.  
	Then we have $\hat{Q}_G= H \cdot Q_G$
	where
	\[H= \prod_{i =1}^{n}\frac{\bigg[ m_i,\ \left\{ \left( m_j,\ \big[ \{ (P_{ji}) \} \big] \right) \;\mid\; j> i \right\}  \bigg] }{\Bigg(\bigg[ m_i,\ \left\{ \left( m_j,\ \big[ \{ (P_{ji}) \} \big] \right) \;\mid\; j > i \right\}  \bigg] ,\bigg [ \{ (P_{si}) \}_{s<i} \bigg] \Bigg) }.
	\]
\end{Lemma}

\begin{proof}
	The result follows directly from Lemma~\ref{lcmgcd}.
\end{proof}

We now prove a main theorem over a PID \(R\).

\begin{Th}
	Let  \((G, \beta)\) be an edge-labeled graph and \(F_1, F_2, \ldots, F_n \in \hat{R}_{G}\) splines. 
	Then the following statements are equivalent:
	\begin{enumerate}
		\item The set \(\{F_1, F_2, \dots, F_n\}\) forms an \(R\)-module basis of \(\hat{R}_{G}\).
		\item The determinant satisfies \(\lvert F_1, F_2, \ldots, F_n \rvert = \pm \hat{Q}_{G}\).
	\end{enumerate}
\end{Th}

\begin{proof}
	The forward direction has been established in Theorem~\ref{GCD basis}.  
	We now prove the converse.  
	Assume that \(\{F_1,F_2, \dots, F_n\}\) is a module basis of \(\hat{R}_G\).  
	By Corallarly ~\ref{det} (see Theroem~6.14 in \cite{DA}), flow-up splines  
	\(\{F^{(1)}, F^{(2)}, \dots, F^{(n)}\}\), where each \(F^{(i)} \in \hat{\mathcal{F}}_i\) is minimal,  
	exist and form a basis such that  
	\(\lvert F^{(1)}, F^{(2)}, \dots, F^{(n)} \rvert =  \pm \hat{Q}_{G}.\)  
	Then, by Corollary~\ref{twobasis}, which states that the determinants of two bases differ at most by a unit (in particular, by sign), it follows that  
	\[
	\lvert F_1, F_2, \dots , F_n \rvert = \pm \lvert F^{(1)}, F^{(2)}, \dots, F^{(n)} \rvert =  \pm \hat{Q}_{G}
	\]
	
\end{proof}

\end{document}